\documentclass[11pt,a4paper,reqno]{amsart}%
\usepackage{amssymb}
\usepackage{amsmath}
\usepackage{amsfonts}
\usepackage{amsthm}
\usepackage{color}
\usepackage[usenames,dvipsnames]{xcolor}
\usepackage{bbm}
\usepackage{graphicx}
\usepackage{rotating}
\usepackage{hyperref}
\usepackage{mathrsfs}
\usepackage{cmll}
\usepackage[all, cmtip]{xy}%
\providecommand{\U}[1]{\protect\rule{.1in}{.1in}}
\newtheorem{theorem}{Theorem}[section]
\newtheorem{corollary}[theorem]{Corollary}
\newtheorem{lemma}[theorem]{Lemma}

\newtheorem*{theorem*}{Theorem}

\theoremstyle{definition}
\newtheorem{definition}[theorem]{Definition}%[subsection]

\theoremstyle{remark}
\newtheorem{remarkx}[theorem]{Remark}
\newenvironment{remark}
  {\pushQED{\qed}\remarkx}
  {\popQED\endremarkx}

\thanks{}
\numberwithin{equation}{section}

\makeatletter
\def\@tocline#1#2#3#4#5#6#7{\relax
  \ifnum #1>\c@tocdepth % then omit
  \else
    \par \addpenalty\@secpenalty\addvspace{#2}%
    \begingroup \hyphenpenalty\@M
    \@ifempty{#4}{%
      \@tempdima\csname r@tocindent\number#1\endcsname\relax
    }{%
      \@tempdima#4\relax
    }%
    \parindent\z@ \leftskip#3\relax \advance\leftskip\@tempdima\relax
    \rightskip\@pnumwidth plus4em \parfillskip-\@pnumwidth
    #5\leavevmode\hskip-\@tempdima
      \ifcase #1
       \or\or \hskip 1.5 em \or \hskip 2em \else \hskip 3em \fi%
      #6\nobreak\relax
    \hfill\hbox to\@pnumwidth{\@tocpagenum{#7}}\par% <---- \dotfill -> \hfill
    \nobreak
    \endgroup
  \fi}
\makeatother

\begin{document}
\newcommand{\R}{{\mathbbm R}}
\newcommand{\C}{{\mathbbm C}} 
\newcommand{\T}{{\mathbbm T}}
\newcommand{\D}{{\mathbbm D}}
\renewcommand{\P}{\mathbb P}

\newcommand{\Aa}{{\mathcal A}}
\newcommand{\Ii}{{\mathbbm K}}
\newcommand{\Jj}{{\mathbbm J}}
\newcommand{\Nn}{{\mathcal N}}
\newcommand{\Ll}{{\mathcal L}}
\newcommand{\Tt}{{\mathcal T}}
\newcommand{\Gg}{{\mathcal G}}
\newcommand{\Dd}{{\mathcal D}}
\newcommand{\Cc}{{\mathcal C}}
\newcommand{\Oo}{{\mathcal O}}

\newcommand{\pr}{\operatorname{pr}}
\newcommand{\bla}{\langle \hspace{-2.7pt} \langle}
\newcommand{\bra}{\rangle\hspace{-2.7pt} \rangle}
\newcommand{\blq}{[ \! [}
\newcommand{\brq}{] \! ]}
 \newcommand{\into}{\mathbin{\vrule width1.5ex height.4pt\vrule height1.5ex}}

\title[Deformation Cohomology of Lie Algebroids and Morita Equivalence]{Deformation Cohomology of Lie Algebroids\\ and Morita Equivalence}

\author{Giovanni Sparano}
\address{DipMat, Universit\`a degli Studi di Salerno, via Giovanni Paolo II n${}^\circ$ 123, 84084 Fisciano (SA) Italy.}
\email{sparano@unisa.it}

\author{Luca Vitagliano}
\address{DipMat, Universit\`a degli Studi di Salerno, via Giovanni Paolo II n${}^\circ$ 123, 84084 Fisciano (SA) Italy.}
\email{lvitagliano@unisa.it}

%\subjclass{22A22, 53D17, 58H15}
%\keywords{Lie algebroid, deformation cohomology, Morita equivalence}

\maketitle

\begin{abstract}
Let $A \Rightarrow M$ be a Lie algebroid. In this short note, we prove that a pull-back of $A$ along a fibration with homologically $m$-connected fibers shares the same deformation cohomology of $A$ up to degree $m$.
\end{abstract}

\tableofcontents

\section{Introduction}

Lie groupoids can be understood as atlases on certain singular spaces,
specifically, \emph{differentiable stacks} and (by the very definition
of stack) two Lie groupoids are Morita equivalent if they give rise to
the same differentiable stack \cite{BX2011}. This means that, when
using Lie groupoids to model differentiable stacks, Morita invariants
describe the intrinsic geometry of the stack. For instance, Lie groupoid
cohomology, and the deformation cohomology of a Lie groupoid are Morita
invariants, but there are more many examples. The terminology is
motivated by the fact that the relationship between a Lie groupoid and
its stack is analogous to the relationship between a ring and its
category of modules.

Lie algebroids are infinitesimal counterparts of Lie groupoids. However,
the former are more general than the latter in the sense that, while all
Lie groupoids \emph{differentiate} to a Lie algebroid, not all Lie
algebroids \emph{integrate} to a Lie groupoid. A consequence of this
is that there is not a notion of Morita equivalence of
Lie algebroids which is universally good, but there are several non-equivalent alternatives. The
weakest (but reasonable) possible one is the
\emph{weak Morita equivalence} introduced by Ginzburg in
\cite{G2001}. For Poisson manifolds, this notion is weaker than Xu's
Morita equivalence \cite{X1991}, but it makes sense for non-integrable
Poisson manifolds.

In a similar way as for Lie groupoids, it is important to identify as
many Morita invariants of Lie algebroids as possible. In
\cite{C2003}, Crainic proves (a statement equivalent to the fact) that
if two Lie algebroids are Morita equivalent in a suitable sense, then
they share the same de Rham cohomology in low degree. In this note, we
prove the analogous result for the deformation cohomology of Lie
algebroids. Notice that, for Lie groupoids, the Morita invariance of Lie
groupoid cohomology has been proved by Crainic himself in
\cite{C2003}, while the deformation cohomology has been introduced, and
its Morita invariance has been proved, only very recently, by Crainic,
Mestre, and Struchiner in \cite{CMS2015}.

We assume that the reader is familiar with Lie algebroids and their
description in terms of graded manifolds. We only recall that a degree
$k$ $\mathbb{N}$-manifold is a graded manifold whose coordinates are
concentrated in non-negative degree up to degree $k$, and an
$\mathbb{N} Q$-manifold is an $\mathbb{N}$-manifold equipped with an
homological vector field. For instance, if $A \Rightarrow M$ is a Lie
algebroid, then shifting by one the degree of the fibers of the vector
bundle $A \to M$, we get a degree $1$ $\mathbb{N} Q$-manifold whose
homological vector field is the de Rham differential $\mathrm{d}_{A}$ of $A$.
Correspondence $A \rightsquigarrow A[1]$ establishes an equivalence
between the category of Lie algebroids and the category
of degree-$1$ $\mathbb{N} Q$-manifolds.

%s1 #&#
\section{The deformation complex of a Lie algebroid}

Let $A \Rightarrow M$ be a Lie algebroid. In degree $k$, the deformation
complex of $A$,
denoted by
$(C_{\mathrm{def}} (A), \delta )$, consists of
$(k+1)$-multiderivations of $A$, i.e.~$\mathbb{R}$-$(k+1)$-linear maps
\begin{equation*}
c : \Gamma (A) \times \cdots \times \Gamma (A) \to \Gamma (A)
\end{equation*}
such that there exists a (necessarily unique) vector bundle map
$s_{c} : \wedge ^{k} A \to TM$ with $c$ and $s_{c}$ satisfying the
following Leibniz rule
\begin{equation*}
c(\alpha _{1}, \ldots , \alpha _{k}, f \alpha _{k+1}) = s_{c} (\alpha
_{1}, \ldots , \alpha _{k})(f) \alpha _{k+1} + f c(\alpha _{1}, \ldots ,
\alpha _{k}, \alpha _{k+1}),
\end{equation*}
for all $\alpha _{1}, \ldots , \alpha _{k+1} \in \Gamma (A)$, and
$f \in C^{\infty }(M)$. The differential $\delta $ is then defined as
\begin{equation*}
\begin{aligned}
\delta c (\alpha _{0}, \ldots , \alpha _{k+1})
& = \sum _{i} (-)^{i} [
\alpha _{i}, c(\alpha _{0}, \ldots , \widehat{\alpha _{i}}, \ldots ,
\alpha _{k+1})] \\
& \quad + \sum _{i<j} (-)^{i+j} c([\alpha _{i}, \alpha
_{j}], \alpha _{0}, \ldots , \widehat{\alpha _{i}}, \ldots ,
\widehat{\alpha _{j}}, \ldots ,\alpha _{k+1}),
\end{aligned}
\end{equation*}
for all $\alpha _{0}, \ldots , \alpha _{k+1} \in \Gamma (A)$. Complex $(C_{\mathrm{def}}(A), \delta)$
appeared for the first time in \cite{Janusz} under the name \emph{complex of multi-quasi-derivations}.
Its cohomology is called the \emph{deformation cohomology} of
$A$ and it is denoted by $H_{\mathrm{def}}(A)$ \cite{CM2008}.

Actually, $C_{\mathrm{def}} (A)$ is not just a complex, but it is the DG
Lie algebra (even more a DG Lie--Rinehart algebra over the de Rham
algebra of $A$) controlling deformations of $A$, in the sense that
\begin{itemize}
\item[$\triangleright $] Lie algebroid structures on $A$ corresponds
bijectively to Maurer--Cartan elements in $C_{\mathrm{def}} (A)$, and
\item[$\triangleright $] if two Lie algebroid structures are isotopic,
the corresponding Maurer--Cartan elements are gauge equivalent, and the
converse is also true when $M$ is compact.
\end{itemize}
There is a simple alternative description of $C _{\mathrm{def}}(A)$ as
the DG Lie algebra of graded derivations of the de Rham algebra
$(C (A), \mathrm{d}_{A})$, where $C (A) = \Gamma ( \wedge ^{\bullet }A^{\ast })$,
and $\mathrm{d}_{A}$ is the usual Lie algebroid de Rham differential. A cochain
$c \in C^{k}_{\mathrm{def}} (A)$ corresponds to the degree $k$
derivation $D_{c}$ mapping $\omega \in C^{l} (A)$ to $D_{c} \omega
\in C^{k+l} (A)$, with
\begin{equation*}
\begin{aligned}
D_{c} \omega (\alpha _{1}, \ldots , \alpha _{k+l})
& =
\sum _{\sigma \in S_{k, l}} (-)^{\sigma }s_{c}(\alpha _{\sigma (1)},
\ldots , \alpha _{\sigma (k)})(\omega (\alpha _{\sigma (k+1)}, \ldots
\alpha _{\sigma (k+ l)})) \\
& \quad - \sum _{\sigma \in S_{k+1, l-1}} (-)^{
\sigma } \omega (c(\alpha _{\sigma (1)}, \ldots , \alpha _{\sigma (k+1)}),
\alpha _{\sigma (k+2)}, \ldots \alpha _{\sigma (k+ l)}).
\end{aligned}
\end{equation*}

When taking this point of view, the Lie bracket in $C
_{\mathrm{def}}(A)$ is just the graded commutator $[-,-]$ of derivations
and $\delta $ is $[\mathrm{d}_{A}, -]$. Finally, we can interpret $C (A)$ as the
DG algebra of smooth functions on the degree $1$ $\mathbb{N} Q$-manifold
$A[1]$, and then cochains in $C_{\mathrm{def}} (A)$ are just vector
fields on $A[1]$:
\begin{equation*}
C (A) = C^{\infty }(A[1]), \quad \text{and} \quad C_{\mathrm{def}} (A)
= \mathfrak{X} (A[1]).
\end{equation*}
In the following, we will mostly take this point of view.

Given two Lie algebroids $A \Rightarrow M$ and $B \Rightarrow N$ and a
Lie algebroid map $F : A \to B$ covering a smooth map $M \to N$, there
is a DG algebra map $F^{\ast }: C (B) \to C (A)$. One can also connect
the deformation complexes as follows. Apply the shift functor to $F$ to
get a map of $\mathbb{N} Q$-manifolds: $F[1] : A[1] \to B[1]$, and
denote by $C (F)$ the space of $F[1]$-relative vector fields,
i.e.~graded $\mathbb{R}$-linear maps $Z : C (B) \to C (A)$ satisfying
the following Leibniz rule
\begin{equation*}
Z(\omega _{1} \wedge \omega _{2}) = Z(\omega _{1}) \wedge F^{\ast }(
\omega ) + (-)^{\vert Z\vert |\omega _{1}|} F^{\ast }(\omega _{1}) \wedge Z(
\omega _{2}).
\end{equation*}
In other words, $C (F) = C (A) \otimes C _{\mathrm{def}} (B)$, where the
tensor product is over $C (B)$, and we changed the scalars via
$F^{\ast }: C (B) \to C (A)$. Yet in other (more geometric) terms,
$C (F)$ consists of sections of the graded vector bundle
$F^{\ast }(TB[1]) \to A[1]$. Clearly, $C (F)$ is a DG $(C (A), C
(B))$-bimodule, whose differential $\delta : C (F) \to C (F)$ is given
by
\begin{equation*}
\delta Z = \mathrm{d}_{A} \circ Z - (-)^{Z} Z \circ \mathrm{d}_{B}, \quad Z \in C(F).
\end{equation*}
Additionally, there are DG module maps:
\begin{equation*}
C_{\mathrm{def}} (A) \mathop{\longrightarrow }^{F_{\star }} C (F)
\mathop{\longleftarrow }^{F^{\star }} C_{\mathrm{def}} (B)
\end{equation*}
given by
\begin{equation*}
F_{\star }: X \mapsto X \circ F^{\ast }, \quad \text{and} \quad F^{
\star }: Y \mapsto F^{\ast }\circ Y.
\end{equation*}

%s2 #&#
\section{Morita equivalence of Lie algebroids}

There is no \emph{universally good} notion of Morita equivalence for Lie
algebroids. Actually, there are several morally similar but inequivalent
definitions, all of which involve the notion of
\emph{pull-back Lie algebroid}. Let $A \Rightarrow M$ be a Lie algebroid
with anchor $\rho : A \to TM$, and let $\pi : P \to M$ be a surjective
submersion. Put
\begin{equation*}
\pi ^{!} A := TP \, {{}_{\textrm{d}\pi } \times _{\rho }}\, A = \{ (v, a) \in TP
\times A : \textrm{d}\pi (v) = \rho (a) \}.
\end{equation*}
Then $\pi ^{!} A$ is a Lie algebroid over $P$ in the following way. First
of all, sections of $\pi ^{!} A \to P$ are pairs $(X, \alpha )$, where
$X$ is a vector field on $P$ and $\alpha $ is a section of the pull-back
bundle $\pi ^{\ast }A \to P$. Additionally $\textrm{d} \pi (X_{p}) = \rho (
\alpha _{p})$ for all $p \in P$. It is easy to see that there exists a
unique Lie algebroid structure $\pi ^{!} A \Rightarrow P$ such that the
anchor $\pi ^{!} A \to TP$ is the projection $(X, \alpha ) \mapsto X$,
and the bracket is
\begin{equation*}
\left [  (X, \pi ^{\ast }\alpha ), (Y, \pi ^{\ast }\beta ) \right ]  =
\left (  [X,Y], \pi ^{\ast }[\alpha , \beta ]\right )
\end{equation*}
on sections of the special form $(X, \pi ^{\ast }\alpha ), (Y, \pi ^{
\ast }\beta )$, with $\alpha , \beta \in \Gamma (A)$. It follows that
the natural projection $\Pi : \pi ^{!} A \to A$ is a Lie algebroid map
(covering $\pi : P \to M$). In particular, there are DG module maps
\begin{equation*}
C_{\mathrm{def}} (\pi ^{!} A)
\mathop{\longrightarrow }^{\Pi _{\star }} C (\Pi )
\mathop{\longleftarrow }^{\Pi ^{\star }} C_{\mathrm{def}} (A).
\end{equation*}
It is worth remarking, for later applications, that there is a short exact
sequence of vector bundles over $P$,
%
%e1 #&#
\begin{equation}
\label{eq:SES}
0 \mathop{\longrightarrow } VP \mathop{\longrightarrow } \pi ^{!} A \mathop{\longrightarrow } \pi ^{\ast }A
\mathop{\longrightarrow } 0,
\end{equation}
where $VP$ is the vertical tangent bundle of $P \to M$, and the
inclusion $VP \hookrightarrow \pi ^{!} A$ maps a vertical vector field
$X$ to $(X, 0)$. Clearly, the projection $\pi ^{!} A \rightarrow \pi ^{
\ast }A$ maps $(X, \alpha )$ to $\alpha $.

%e2.1 #&#
\begin{definition}[Ginzburg \cite{G2001}]
\label{def:Morita_eq}
Two Lie algebroids $A \Rightarrow M$ and $B \Rightarrow N$ are
(\emph{weak}) \emph{Morita equivalent} if there exist surjective
submersions
\begin{equation*}
M \mathop{\longleftarrow }^{\pi } P \mathop{\longrightarrow }^{\tau }
N
\end{equation*}
with simply connected fibers, such that the pull-back Lie algebroids
$\pi ^{!} A$ and $\tau ^{!} B$ are isomorphic.
\end{definition}

%r1 #&#
\begin{remark}
Let $A \Rightarrow M$ be a Lie algebroid, $P \to M$ be a surjective
submersion and let $E \to M$ be a vector bundle carrying a
representation of $A$. Then $\pi ^{\ast }E \to P$ carries a
representation of the pull-back Lie algebroid $\pi ^{!} A$. Definition~\ref{def:Morita_eq}, originally due to Ginzburg, is then motivated by
the fact that, if $\pi $ has simply connected fibers, correspondence
$E \rightsquigarrow \pi ^{\ast }E$ establishes an equivalence between the
categories of $A$-representations and of $\pi ^{!} A$-representations.
However, other reasonable definitions of Morita equivalence are possible.
For instance, one could require submersions $\pi , \tau $ to have
fibers with specific, higher connectedness (or even cohomological
connectedness) properties.
\end{remark}

%r2 #&#
\begin{remark}
Let $A \Rightarrow M$ and $B \Rightarrow N$ be weak Morita equivalent
Lie algebroids, and let $A \leftarrow P \rightarrow B$ be surjective
submersions realizing the equivalence. Then $A$ and $B$ share several
properties. For instance, there is a bijection between their leaf
spaces, and corresponding leaves have the same fundamental group.
Additionally, the Lie algebroid structures \emph{transverse} to
corresponding leaves are isomorphic and so are the stabilizer of $A$ at
$x \in M$ and the stabilizer of $B$ at $y \in N$ whenever $x,y$ are projections of the
same point in $P$ \cite{G2001}. Finally, $A$ and $B$ share the same
$0$-th and $1$-st de Rham cohomology (see also Theorem~\ref{theor:dR_cohom} below).
\end{remark}

We conclude this section by describing the de Rham complex of a pull-back
Lie algebroid $\pi ^{!} A \Rightarrow P$. To do this, we will interpret
$C (A)$ as a DG algebra over differential forms on $M$ via the pull-back
$\rho ^{\ast }: \Omega (M) \to C (A)$ along the anchor.

%l2.2 #&#
\begin{lemma}
\label{lemma}
There are canonical isomorphisms of DG modules
%
%e2 #&#
\begin{equation}
\label{eq:isom}
C (\pi ^{!} A) = \Omega (P) \!\mathop{\otimes }_{\Omega (M)}\! C (A)
\quad \text{and} \quad C (\Pi ) = \Omega (P) \!\mathop{\otimes }_{\Omega (M)}\! C _{\mathrm{def}}(A).
\end{equation}
\end{lemma}

\begin{proof}
The simplest proof is via graded geometry. Consider the pull-back
diagram
\begin{equation*}
\begin{array}{c}
\xymatrix{ \pi ^! A \ar[r]^-\rho \ar[d]_-{\Pi } & TP \ar[d]^-{\textrm{d}\pi } \\
A \ar[r]^-\rho & TM }
\end{array}
.
\end{equation*}
The top row consists of Lie algebroids over $P$, while the bottom row
consists of Lie algebroids over $M$. Shifting by $1$ the degree in the
fibers of all of them, we get a diagram of
$\mathbb{N} Q$-manifolds
%
%e3 #&#
\begin{equation}
\label{eq:shift_diag}
\begin{array}{c}
\xymatrix{ \pi ^! A[1] \ar[r] \ar[d] & T[1]P \ar[d] \\
A[1] \ar[r] & T[1] M }
\end{array}
.
\end{equation}
It follows from the functorial properties of the shift that
(\ref{eq:shift_diag}) is a pull-back diagram as well. Hence functions
on $\pi ^{!} A[1]$ are the tensor product over functions on $T[1] M$ of
the functions on $T[1] P$ and the functions on $A[1]$. This is
precisely the content of the first isomorphism in (\ref{eq:isom}). The
second isomorphism immediately follows from the first one and the fact
that $\Pi [1]$-relative vector fields are the tensor product over
functions on $A[1]$ of vector fields on $A[1]$ and functions on
$\pi ^{!} A [1]$.
\end{proof}

%s3 #&#
\section{Morita invariance of the deformation cohomology}

The de Rham cohomologies of Lie algebroids are Morita invariant. More
precisely, we have the following theorem due to Crainic.

%t3.1 #&#
\begin{theorem}[Crainic \cite{C2003}]
\label{theor:dR_cohom}
Let $A \Rightarrow M$ be a Lie algebroid and let $\pi : P \to M$ be a
surjective submersion with homologically $m$-connected fibers, then
$A$ and the pull-back Lie algebroid $\pi ^{!} A$ share the same de Rham
cohomology up to degree $m$. Specifically, the graded vector space map
$\Pi ^{\ast }: H_{\mathrm{dR}}(A) \to H_{\mathrm{dR}}( \pi ^{!} A)$ is an
isomorphism in degree $q \leq m$.
\end{theorem}

\begin{proof}
We briefly recall Crainic's proof. This will be useful in the following.
Begin noticing that the inclusion $VP \hookrightarrow \pi ^{!} A$ is the
inclusion of a Lie subalgebroid. Accordingly, there is a distinguished
subcomplex, and an ideal, $F_{1} C \subset C(\pi ^{!} A)$ consisting of
cochains vanishing when applied to sections of $VP$. Hence $C(\pi ^{!}
A)$ is canonically equipped with a filtration
\begin{equation*}
C(\pi ^{!} A) = F_{0} C \supset F_{1} C \supset \cdots \supset F_{p} C
\supset \cdots
\end{equation*}
where $F_{p} C $ is the $p$-th exterior power of $F_{1} C$. We denote
by $E$ the associated (first quadrant) spectral sequence which
computes $H_{\mathrm{dR}} (\pi ^{!} A)$. It follows from the short
exact sequence (\ref{eq:SES}) that
\begin{equation*}
E_{0}^{p,q} = F_{p} C^{p+q} / F_{p+1} C^{p+q} = V\Omega ^{q} \otimes C
^{p} (A)
\end{equation*}
where $V\Omega = \Gamma (\wedge ^{\bullet }V^{\ast }P)$ are vertical
differential forms on $P$, and the tensor product is over $C^{\infty
}(M)$. Now, it is easy to see that differential
\begin{equation*}
\mathrm{d}_{0}^{p, \bullet } : V\Omega ^{\bullet }\otimes C^{p} (A) \to V
\Omega ^{\bullet + 1} \otimes C^{p} (A)
\end{equation*}
is just the vertical de Rham differential $\mathrm{d}^{V}$ (up to tensoring by
$C^p (A)$), and, from the connectedness hypothesis, we have $H^{0}(V
\Omega , \mathrm{d}^{V}) = C^{\infty }(M)$, and $H^{q}(V\Omega , \mathrm{d}^{V}) = 0$ for
$0 < q \leq m$. Hence $E_{1}^{\bullet , 0} = C(A)$, and
\begin{equation*}
E_{1}^{\bullet , q} = H^{q} ( V\Omega ^{\bullet }\otimes C (A), \mathrm{d}^{V})
= 0 \quad \text{for $0 < q \leq m$}.
\end{equation*}
Additionally $\mathrm{d}_{1}^{\bullet , 0} : C(A) \to C(A)$ is precisely the de
Rham differential $\mathrm{d}_{A}$. We conclude that
\begin{equation*}
E_{2}^{p, q} = E_{\infty }^{p, q} = 0 \quad \text{for $1 < p + q
\leq m$},
\end{equation*}
and\nopagebreak
\begin{equation*}
H^{q}_{\mathrm{dR}}(A) = E_{2}^{0, q} = E_{\infty }^{0, q} = H^{q}
_{\mathrm{dR}} (\pi ^{!} A) \quad \text{for $q \leq m$}.
\end{equation*}

\end{proof}

We now come to our main result. The following theorem (and its
corollary) is our version of the Morita invariance of the deformation
cohomology.

%t3.2 #&#
\begin{theorem}
\label{theor:def_cohom}
Let $A \Rightarrow M$ be a Lie algebroid and let $\pi : P \to M$ be a
surjective submersion with homologically $m$-connected fibers, then
$A$ and the pull-back Lie algebroid $\pi ^{!} A$ share the same
deformation cohomology up to degree $m$.
\end{theorem}

\begin{proof}
The present proof is inspired by the Crainic, Mestre, and Struchiner proof
of the Morita invariance of the deformation cohomology of Lie groupoids
\cite{CMS2015}. However, notice that, in our statement, the Lie
algebroid $A$ needs not to be integrable. Consider the DG module maps
\begin{equation*}
C_{\mathrm{def}} (\pi ^{!} A)
\mathop{\longrightarrow }^{\Pi _{\star }} C (\Pi )
\mathop{\longleftarrow }^{\Pi ^{\star }} C_{\mathrm{def}} (A).
\end{equation*}
Our strategy consists in proving that
\begin{enumerate}
\item[(i)]
$\Pi _{\star }$ is a quasi-isomorphism (regardless the connectedness
properties of the fibers of $\pi $);
\item[(ii)]
$\Pi ^{\star }$ induces an isomorphism in cohomology up to degree $m$.
\end{enumerate}
We begin with (i). As $\pi $ is a surjective submersion, then
$\Pi _{\star }$ is surjective. Actually it consists in restricting a
derivation of $C(\pi ^{!} A)$ to the DG subalgebra $C(A) \hookrightarrow
C(\pi ^{!} A)$. Geometrically, it consists in composing a vector field
on $\pi ^{!} A [1]$ with projection $\textrm{d} \Pi [1] : T \pi ^{!} A[1]
\to TA[1]$. Denote $\mathcal{K} := \ker \Pi _{\star }$ and consider the
short exact sequence of DG modules
\begin{equation*}
0 \mathop{\longrightarrow } \mathcal{K} \mathop{\longrightarrow } C_{\mathrm{def}} (\pi ^{!} A) \mathop{\longrightarrow }^{\Pi _{\star }} C(\Pi ) \mathop{\longrightarrow } 0.
\end{equation*}
It is enough to show that $\mathcal{K}$ is acyclic. To do this, we
construct a contracting homotopy $h : \mathcal{K} \to \mathcal{K}$ for
$(\mathcal{K}, \delta )$. Notice that $\mathcal{K}$ consists of
derivations of $C(\pi ^{!} A)$ vanishing on $C(A)$, equivalently it
consists of vector fields on $\pi ^{!} A[1]$ that are the vertical wrt
projection $\Pi [1] : \pi ^{!} A[1] \to A[1]$. As (\ref{eq:shift_diag})
is a pull-back diagram, the vector fields in $\mathcal{K}$ are completely
determined by their composition with $T \pi ^{!} A[1] \to T T[1] P$. This
shows that there is a DG module isomorphism
\begin{equation*}
\mathcal{K} \mathop{\longrightarrow }^{\simeq } C(\pi ^{!} A) \!\mathop{\otimes }_{
\Omega (P)}\! \mathcal{V} = C(A) \!\mathop{\otimes }_{\Omega (M)}\! \mathcal{V}
\end{equation*}
where $\mathcal{V}$ is the DG $\Omega (P)$-module of vector fields on
$T[1] P$ that are vertical with respect to projection $T[1] P \to T[1] M$. Now we
construct a contracting homotopy $h_{\mathcal{V}} : \mathcal{V}
\to \mathcal{V}$ for $\mathcal{V}$. First recall that the differential
in $\mathcal{V}$ is the commutator with the de Rham differential
$\mathrm{d} : \Omega (P) \to \Omega (P)$. Now, every vector field $V$ on
$T[1] P$ is a derivation of $\Omega (P)$. Hence, it can be uniquely
written in the form $V = \mathcal{L}_{J} + i_{K}$, where $J, K$ are form
valued vector fields on $P$. Additionally, $V$ is vertical with respect to
$T[1] P \to T[1] M$ if and only if $J, K$ are both vertical wrt $P \to M$. Put
$h_{\mathcal{V}} (V) := (-)^{|V|} i_{J}$, so that $V = [\mathrm{d}, h_{
\mathcal{V}} (V)] + h_{\mathcal{V}}([\mathrm{d}, V])$, showing that $h_{
\mathcal{V}}$ is indeed a contracting homotopy. It is easy to see that
$h_{\mathcal{V}}$ is $\Omega (M)$-linear and we define
\begin{equation*}
h : C(A) \!\mathop{\otimes }_{\Omega (M)}\! \mathcal{V} \to C(A) \!\mathop{\otimes }_{
\Omega (M)}\! \mathcal{V}, \quad \omega \otimes V \mapsto (-)^{|
\omega |} \omega \otimes h_{\mathcal{V}}(V).
\end{equation*}
A straightforward computation shows that $h$ is a contracting homotopy.
Hence $\mathcal{K}$ is acyclic.

Now we prove (ii). The proof is analogous to that of Theorem~\ref{theor:dR_cohom}. As $C (\Pi )$ is a DG $C(\pi ^{!} A)$-module, there
is a filtration
\begin{equation*}
C(\Pi ) = F_{0} C \cdot C(\Pi ) \supset F_{1} C \cdot C(\Pi ) \supset
\cdots \supset F_{p} C \cdot C(\Pi ) \supset \cdots
\end{equation*}
hence a spectral sequence computing the cohomology of $C(\Pi )$, that
we denote again by $E$. From Lemma~\ref{lemma} and from (\ref{eq:SES})
again, we have
\begin{equation*}
E_{0}^{p,q} = F_{p} C^{p+q} \cdot C(\Pi ) / F_{p+1} C^{p+q} \cdot C(
\Pi ) = V\Omega ^{q} \otimes C_{\mathrm{def}}^{p} (A),
\end{equation*}
where the tensor product is over $C^{\infty }(M)$, and the differential
$\mathrm{d}_{0} : E_{0} \to E_{0}$ is the vertical de Rham differential
$\mathrm{d}^{V} : V\Omega \to V \Omega $ (tensorized by $C_{\mathrm{def}}^{
\bullet }(A)$). Now the proof proceeds exactly as the proof of
Theorem~\ref{theor:dR_cohom}, and we leave the details to the reader.
\end{proof}

%c3.3 #&#
\begin{corollary}
Let $A \Rightarrow M$ and $B \Rightarrow N$ be (weak) Morita equivalent
Lie algebroids. Then $A$ and $B$ share the same $0$-th and $1$-st
deformation cohomology. If, additionally, the Morita equivalence is
realized by surjective submersions $M \leftarrow P \rightarrow N$ with
homologically $m$-connected fibers, then $A$ and $B$ share the same
deformation cohomology up to degree $m$.
\end{corollary}

%s4 #&#
\section{An illustrative example: deformations of weak Morita equivalent foliations}

Let $M$ be a manifold and let $\mathcal{F}$ be a foliation of $M$.
The deformations of $\mathcal{F}$ are controlled by $\Omega (\mathcal{F} , TM/T
\mathcal{F})$: leafwise differential
forms with values in the Bott representation  \cite{H1973}. Cochain complex $\Omega (\mathcal{F} ,
TM/T \mathcal{F})$ is a deformation retract of the deformation complex
$C_{\mathrm{def}} (T\mathcal{F})$ \cite{CM2008,V2014}. In particular,
$\Omega (\mathcal{F} , TM/T \mathcal{F})$ and $C_{\mathrm{def}} (T
\mathcal{F})$ share the same cohomology. This means that, morally,
deforming a foliation $\mathcal{F}$ or its tangent algebroid
$T\mathcal{F}$ is the same. This should be expected from the fact that
small deformations of $T \mathcal{F}$ preserve the injectivity of the
anchor.

Now let $\mathcal{V} \subset \mathcal{H}$ be a flag of foliations of a
manifold $P$. In other words, $\mathcal{V}$ and $\mathcal{H}$ are
foliations, and the leaves of $\mathcal{V}$ are contained into leaves
of $\mathcal{H}$. Yet in other terms $T \mathcal{V} \subset T
\mathcal{H}$. Assume that $\mathcal{V}$ is simple, i.e.~its leaf space
$M$ is a manifold and the projection $\pi : P \to M$ is a surjective
submersion. In other words, $\mathcal{V} = VP$: the vertical bundle of
$P$ with respect to $\pi $. From involutivity
$\pi _{\ast }(T \mathcal{H}) = T \mathcal{F}$ for a, necessarily unique,
foliation $\mathcal{F}$ of $M$, and $T \mathcal{H} = (\textrm{d}\pi )^{-1} (T
\mathcal{F})$. It is then immediate to see that $T \mathcal{H} = \pi
^{!} T \mathcal{F}$: the pull-back Lie algebroid. So, if $P$ has
homologically $m$-connected fibers, $T \mathcal{F}$ and $T
\mathcal{H}$ share the same deformation cohomology up to degree $m$.
Now, using that $\Omega (\mathcal{F} , TM/T \mathcal{F})$ is a
deformation retract of $C_{\mathrm{def}} (T \mathcal{F})$ (and similarly
for $\mathcal{H}$) we immediately get the following theorem.

%t4.1 #&#
\begin{theorem}
\label{theor:fol}
Let $\mathcal{V} \subset \mathcal{H}$ be a flag of foliations on $P$.
Assume that $\mathcal{V}$ is simple and let $\mathcal{F}$ be the
foliation induced by $\mathcal{H}$ on the leaf space of $\mathcal{V}$
via projection. If the leaves of $\mathcal{V}$ are $m$-simply connected,
then $\mathcal{F}$ and $\mathcal{H}$ share the same deformation
cohomology up to degree $m$.
\end{theorem}

It is now natural to define a weak Morita equivalence for foliated
manifolds. Namely, two foliated manifolds $(M, \mathcal{F})$ and
$(N, \mathcal{G})$ are \emph{weak Morita equivalent} if the tangent
algebroids of $\mathcal F$ and $\mathcal G$ are weak Morita equivalent. Theorem~\ref{theor:fol} then
reveals that if $(M, \mathcal{F})$ and $(N, \mathcal{G})$ are Morita
equivalent, then $\mathcal{F}$ and $\mathcal{G}$ share the same $0$-th
and $1$-st deformation cohomology. If, additionally, the Morita
equivalence is realized by surjective submersions with homologically
$m$-connected fibers, then $\mathcal{F}$ and $\mathcal{G}$ share the
same deformation cohomology up to degree $m$.

Finally, notice that the tangent algebroid of a foliation $
\mathcal{F}$ is always integrable. Any integration of $\mathcal{F}$ is
called a foliation groupoid \cite{CM2001}. It would be interesting to
explore the relationship between the weak Morita equivalence of
foliations and the Morita equivalence of their foliation groupoids
(particularly the monodromy and the holonomy groupoids). However, this
goes beyond the scopes of the present note.

\section{Final remarks}

There is another approach to the deformation cohomology of a Lie
algebroid. Namely, the deformation complex of a Lie algebroid
$A \Rightarrow M$ can be seen as the \emph{linear de Rham complex} of
the cotangent VB-algebroid $(T^{\ast }A \Rightarrow A^{\ast }) \to (A
\Rightarrow M)$ \cite{GSM2010,AAC2012}. Recall that a VB-algebroid is
a vector bundle object in the category of Lie algebroids, and its linear
de Rham complex is the subcomplex in the de Rham complex consisting of
cochains that are linear with respect to the vector bundle structure. It is possible
to define a notion of (weak) Morita equivalence for VB-algebroids
respecting the vector bundle structure. It is then natural to expect
that (1) if two Lie algebroids are (weak) Morita equivalent, then their
cotangent VB-algebroids are (weak) Morita equivalent, and (2) if two
VB-algebroids are (weak) Morita equivalent, then their linear de Rham
cohomologies are the same in low degree. If so, then Theorem~\ref{theor:def_cohom} would be an immediate corollary. This alternative
approach to the Morita invariance of the deformation cohomology of Lie
algebroids is actually being investigated in a separate work
\cite{LV2018}. We remark that Morita equivalence for VB-groupoids,
i.e.~vector bundle objects in the category of Lie groupoids, is defined
and discussed in \cite{dHO2016}, where the authors prove the Morita
invariance of the VB-groupoid cohomology. Finally, VB-groupoid
cohomology is related to VB-algebroid cohomology by a Van-Est type map
\cite{CD2017}.

\section*{Acknowledgments} LV is member of the GNSAGA of INdAM.

\end{document}